\documentclass[12pt]{amsart}

\usepackage{amsrefs}
\usepackage{amssymb,mathtools}
\usepackage{fancyhdr}
\usepackage{bbm}
\usepackage{xcolor}
\usepackage{hyperref}
\hypersetup{
  colorlinks=true,
  linkcolor=blue,
  citecolor=blue
}
\usepackage{comment}

\calclayout

\newtheorem{theorem}{Theorem}
\newtheorem*{theorem*}{Theorem}
\newtheorem{lemma}[theorem]{Lemma}
\newtheorem{proposition}[theorem]{Proposition}
\newtheorem{claim}[theorem]{Claim}
\newtheorem{corollary}[theorem]{Corollary}

\newtheorem*{question*}{Question}
\newtheorem{remark}[theorem]{Remark}

\newtheorem{maintheorem}{Theorem}

\theoremstyle{definition}

\newtheorem*{definition*}{Definition}
\newtheorem*{lemma*}{Lemma}
\usepackage{graphicx}
\usepackage{pstricks, enumerate, pst-node, pst-text, pst-plot}

\numberwithin{equation}{section}
\numberwithin{theorem}{section}

\newcommand{\R}{\mathbb{R}}

\newcommand{\Z}{\mathbb{Z}}

\newcommand{\eps}{\varepsilon}

\usepackage{xparse}
\DeclareDocumentCommand\Pr{ m g }{\ensuremath{
    {   \IfNoValueTF {#2}
      {\mathbb{P}\left[{#1}\right]}
      {\mathbb{P}\left[{#1}\middle\vert{#2}\right]}%
    }
}}
\DeclareDocumentCommand\E{ m g }{\ensuremath{
    {   \IfNoValueTF {#2}
      {\mathbb{E}\left[{#1}\right]}
      {\mathbb{E}\left[{#1}\middle\vert{#2}\right]}%
    }
}}

\def\ee{\mathrm{e}}

\begin{document}

\title[]{Asymptotic R\'enyi Entropies of Random Walks on Groups}

\author[]{Kimberly Golubeva}
\author[]{Minghao Pan}
\author[]{Omer Tamuz}
\address{California Institute of Technology}

\thanks{This work was supported by a National Science Foundation CAREER award (DMS-1944153). The first author was supported by the Canadian Natural Sciences and Engineering Research Council (NSERC) PGS-D Scholarship [funding reference number 567723]. }
\date{\today}

\begin{abstract}

We introduce asymptotic R\'enyi entropies as a parameterized family of invariants for random walks on groups. These invariants interpolate between various well-studied properties  of the random walk, including the growth rate of the group, the Shannon entropy, and the spectral radius. They furthermore offer large deviation counterparts  of the Shannon-McMillan-Breiman Theorem. We prove some basic properties of asymptotic R\'enyi entropies that apply to all groups, and discuss their analyticity and positivity for the free group and lamplighter groups.


\end{abstract}

\maketitle
\section{Introduction}

The Avez entropy (or asymptotic Shannon entropy) of a random walk on a group is an essential tool for understanding its asymptotic properties, and in particular the Furstenberg-Poisson boundary \cite{avez1974theorem, kaimanovich1983random}. It is also useful for studying geometric properties of groups; for example,  it is always positive for non-amenable groups, and zero for sub-exponential groups.

We introduce asymptotic R\'enyi entropies of a random walk on a finitely generated group. This is a family of invariants that generalizes the Avez entropy, as well as other useful invariants such as the growth rate of the group and the spectral radius of the walk. R\'enyi entropies originated in information theory as a general way to quantify randomness, beyond Shannon entropy \cite{renyi1961measures}. They share some (but not all) of the useful properties of the Shannon entropy, including additivity for product measures and monotonicity under push-forwards, which makes them useful in the setting of  random walks on groups.\footnote{See \cite{csiszar2008axiomatic,principe2010information,mu2021blackwell} for axiomatic treatments of R\'enyi entropies and the related R\'enyi divergences. The axiomatization in the latter implies that R\'enyi divergences are the extreme points in the set of all divergences that are additive and monotone under push-forwards. A similar result applies to R\'enyi entropies.}

Let $\nu$ be a finitely supported probability distribution on a countable set $\Omega$. For $\alpha \in (0,\infty) \setminus \{1\}$, the {\em $\alpha$-R\'enyi entropy} of $\nu$ is given by
\begin{align*}
    H_\alpha(\nu) = \frac{1}{1-\alpha}\log\sum_{\omega \in \Omega}\nu(\omega)^\alpha.
\end{align*}
Letting $H_1(\nu)$ be the Shannon entropy makes $\alpha \mapsto H_\alpha(\nu)$ a continuous map at $\alpha=1$. Likewise, letting $H_0(\nu)$ be the logarithm of the size of the support and $H_\infty(\nu) = -\max_\omega\log \nu(\omega)$ extends this map to a continuous one defined on the domain $[0,\infty]$.

Let $\mu$ be a finitely supported probability measure on a group $G$, and denote by $\mu^{(n)}$ the $n$-fold convolution of $\mu$. We say that $\mu$ is non-degenerate if its support generates $G$ as a semigroup. For $\alpha \in [0,\infty]$, the \emph{asymptotic $\alpha$-R\'enyi entropy} of the $\mu$-random walk on $G$ is
\begin{align*}
    h_\alpha(\mu) = \lim_{n \to \infty}\frac{1}{n}H_\alpha(\mu^{(n)}).
\end{align*}
As we explain, this limit exists for every $\mu$. For non-degenerate $\mu$, it is easy to see that $h_0(\mu)$ is the exponential growth rate of $G$,  $h_1(\mu)$ is the Avez entropy, and  that for symmetric random walks, $h_\infty(\mu)$ is minus the logarithm of the spectral radius (Claim~\ref{clm:spectral}).

We begin by establishing some general properties that apply to all finitely supported measures on groups. 
\begin{maintheorem}
\label{thm:main-cont}
    Let $\mu$ be a finitely supported probability measure on a group $G$, and consider the map $\alpha \mapsto h_\alpha(\mu)$.

    \begin{enumerate}
        \item For $\alpha$ in $[0,1]$, $h_\alpha(\mu)$ is continuous and decreasing. It is strictly decreasing for all $\alpha$ such that $h_\alpha(\mu) > h_1(\mu)$. 
        \item For $\alpha \in (1,\infty]$, $h_\alpha(\mu)$ is continuous and decreasing.
        \item For symmetric $\mu$ and $\alpha \in (2,\infty)$, $h_\alpha(\mu) = \frac{\alpha}{\alpha-1}h_\infty(\mu)$.
    \end{enumerate} 
\end{maintheorem}
A few observations are in order. Part (1) implies that $h_\alpha(\mu)$ interpolates continuously between the exponential growth rate $h_0(\mu)$ and the Avez entropy $h_1(\mu)$, yielding a non-trivial family of invariants indexed by $\alpha \in [0,1]$. Parts (1) and (2) together imply that $h_\alpha(\mu)$ is everywhere (weakly) decreasing, and continuous except possibly at $\alpha=1$. As we shall see, it is possible to have a discontinuity there. In particular, there will be a discontinuity for every symmetric, positive entropy random walk on an amenable group. We also note that by Theorem~\ref{thm:main-free} below, while $h_\alpha$ is continuous on $(1,\infty)$, it is not always twice-differentiable in this range. We do not know if it is always differentiable. For symmetric random walks, part (3) shows that asymptotic R\'enyi entropies have a trivial form in the range $\alpha \geq 2$. The proof of part (1) uses two convexity properties of R\'enyi entropies (under re-parameterization), including a novel one which we show in  Proposition~\ref{prop:convex}.

\begin{figure}
    \centering
    \includegraphics[scale=0.75]{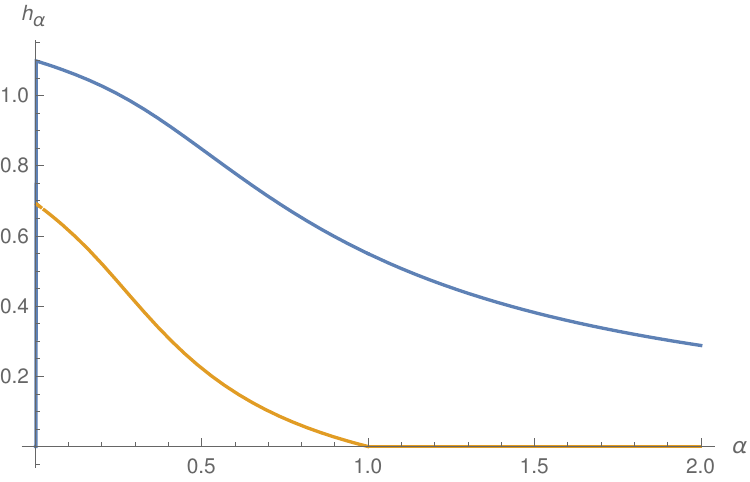}
    \caption{R\'enyi entropies for the simple random walk on the free group with two generators (blue, higher) and the switch-walk-switch walk on the lamplighter group (orange, lower). For both, $h_0$ is the exponential growth rate of the group, and $h_1$ is the Avez entropy. Elementary formulas for these graphs are given in \eqref{eq:free_entropy} and \eqref{eq:sws}.
    \label{fig:entropy}}
\end{figure}

Figure~\ref{fig:entropy} shows the asymptotic R\'enyi entropies of the simple random walk on the free group and the switch-walk-switch walk on the lamplighter group; we calculate these explicitly below. This graph illustrates some properties of asymptotic R\'enyi entropies that hold more generally: (i) $h_\alpha$ is (weakly) decreasing in $\alpha$, (ii) for symmetric random walks on non-amenable groups, $h_\alpha$ is positive for all $\alpha \in [0,\infty]$, (iii) for symmetric random walks on amenable groups, $h_\alpha$ vanishes on $(1,\infty]$ (Corollary~\ref{claim:amenable}). In both of these graphs, $h_\alpha$ is continuous. Theorem~\ref{thm:main-cont} shows that this is generally the case, except perhaps at $\alpha=1$. 

\subsection*{The asymptotic min-entropy}
The R\'enyi entropy $H_\infty$ is known as the min-entropy. The asymptotic min-entropy is $h_\infty = \lim_n-\frac{1}{n}\log\max_g \mu^{(n)}(g)$ is the exponential rate of decay of the largest atom in $\mu^{(n)}$. For symmetric random walks it is well-known that this maximum is achieved at the identity (at even times), and that $h_\infty$ is minus the logarithm of the norm of the Markov operator (see, e.g., \cite[Proposition 4.4.9]{lalley2023random}). Thus---for symmetric non-degenerate random walks---it follows from Kesten's Theorem~\cite{kesten1959symmetric} that $h_\infty$ vanishes if and only if the group is amenable. For non-symmetric random walks, we give an example of a walk on an amenable group for which $h_\infty>0$ (Claim~\ref{clm:h-infty-amenable}). For non-amenable groups we show that  $h_\infty>0$ for every non-degenerate random walk, including the non-symmetric ones (Claim~\ref{clm:h-infty-nonamenable}).


\subsection*{The log-likelihood process}
Fix a group $G$ and a finitely supported random walk $\mu$. Let $X_1,X_2,\ldots$ be random variables distributed i.i.d.\ $\mu$, and let $Z_n = X_1 \cdot X_2 \cdots X_n$, so that $Z_n$ has distribution $\mu^{(n)}$, and $Z_1,Z_2,\ldots$ is the $\mu$-random walk on $G$. Define the \emph{log-likelihood} process by $L_n = -\log\mu^{(n)}(Z_n)$. Then the Shannon entropy of $Z_n$ is the expectation of $L_n$, and the Shannon-McMillan-Breiman Theorem is the SLLN for the process $(L_n)_n$. 

R\'enyi entropies are, up to a reparametrization, the cumulant generating function of $L_n$:
\begin{align*}
    K_{L_n}(t) = \log\E{\ee^{t L_n}} = t H_{1-t}(\mu^{(n)}).
\end{align*}
Hence the asymptotic R\'enyi entropies of the random walk capture the asymptotics of the moment generating functions of the log-likelihoods. It follows that  the R\'enyi entropies are useful for establishing large deviation bounds for $\frac{1}{n}L_n$. In particular,  whenever $(1-\alpha)h_\alpha(\mu)$ is strictly convex its Legendre transform yields a rate function for large deviations of $\frac{1}{n}L_n$. Likewise, positivity of $h_\alpha(\mu)$ yields a Chernoff bound.

These observations lead us to study the positivity and convexity of asymptotic R\'enyi entropies. Likewise, we are interested in deviations from continuity and analyticity, as representing phase transitions.

\subsection*{Positivity of asymptotic R\'enyi entropies} For symmetric random walks on amenable groups, we use standard results to show that $h_\alpha(\mu)=0$ whenever $\alpha>1$. For non-amenable groups and non-degenerate $\mu$, regardless if the random walks are symmetric or not, $h_\alpha(\mu) > 0$ for all $\alpha \in [0,\infty]$.

Amenable groups with exponential growth will have $h_0(\mu)>0$ for all non-degenerate $\mu$. If the Avez entropy $h_1(\mu)$ is also positive, then $h_\alpha(\mu)$ will be positive on $[0,1]$, since $h_\alpha(\mu)$ is decreasing. However, if $h_1(\mu)=0$ then it is possible that $h_\alpha(\mu)$ vanishes already for  some $\alpha < 1$. We offer this as an open question:
\begin{question*}
Does there exist a non-degenerate $\mu$ on a group with an exponential growth rate for which $h_\alpha(\mu)=0$ for some $\alpha\in (0,1)$?    
\end{question*}
 While a natural candidate would be the lamplighter group, we show that this is not the case. 
\begin{maintheorem}\label{thm:main-ll-positivity}
For any non-degenerate,  symmetric, finitely supported probability measure $\mu$ on the lamplighter group $G = L \wr \Z$  with lamps in a non-trivial finite group $L$, $h_\alpha(\mu)>0$ for all $\alpha\in (0,1)$.
\end{maintheorem}
The proof of Theorem~\ref{thm:main-ll-positivity} involves ``tilting'' $\mu$ along the second coordinate, resulting in a random walk that has positive drift, and then relating its R\'enyi entropies to those of the original, untilted walk.\footnote{For $\mu$-random walks with positive drift in the second coordinate (the location of the lamplighter) it is known that $h_1(\mu)>0$ and hence, by the monotonicity of R\'enyi entropies, the claim also follows. For asymmetric $\mu$ with zero drift the claim is likewise true, using a similar proof to that of our Theorem~\ref{thm:main-ll-positivity}.}

\bigskip 

\subsection*{Continuity and analyticity of asymptotic R\'enyi entropies}
Theorem~\ref{thm:main-cont} states that $h_\alpha$ is continuous everywhere except perhaps at $\alpha=1$. For symmetric random walks on amenable groups, since $h_\alpha=0$ for all $\alpha>1$,  $h_\alpha$ is discontinuous at $\alpha=1$ if and only if the Avez entropy $h_1$ is positive. It is natural to next ask for which random walks on non-amenable groups is $h_\alpha$ continuous at $\alpha=1$. In particular, one may conjecture that this holds for any random walk on a hyperbolic group (see, e.g., \cite{boulanger2021large} for results in this spirit). 

As an example supporting this conjecture, we show that for the simple random walk on the free group, $h_\alpha$ is continuous at $\alpha=1$. More generally, we show that $h_\alpha$ is (mostly) analytic.
\begin{maintheorem}
\label{thm:main-free}
Let $\mu$ be the simple random walk on a free group with at least two generators. Then $h_\alpha(\mu)$ is analytic on $(0,\infty)\setminus \{2\}$, but not at $\alpha=2$, where it is not twice-differentiable.
\end{maintheorem}

As another example we study the analyticity of the R\'enyi entropies of the ``switch-walk-switch'' (SWS) random walk on the lamplighter group $\mathbb{Z}_2 \wr \Z$. Since the Avez entropy vanishes for this walk, we know that $h_\alpha(\mu)=0$ for all $\alpha \geq 1$.
\begin{maintheorem}
\label{thm:main-ll}
Let $\mu$ be the SWS random walk on $\mathbb{Z}_2 \wr \Z$. Then $h_\alpha(\mu)$ is analytic on $(0,\infty)\setminus \{1\}$, but not at $\alpha=1$.
\end{maintheorem}
We prove Theorems~\ref{thm:main-free} and~\ref{thm:main-ll} by explicitly calculating the R\'enyi entropies and showing that they are elementary functions. 


We have so far considered only finitely supported $\mu$. We end this section with a note about non-finitely supported $\mu$. For such $\mu$, $H_0(\mu) = \infty$, and thus $h_0(\mu) = \infty$. Nevertheless, it is still possible that $h_\alpha(\mu) < \infty$ for some $\alpha >0$. For example, infinitely supported random walks with finite Shannon entropy are important in the study of groups of subexponential growth (see, e.g., \cite{frisch2019choquet, erschler2020growth}). For $\alpha > 1$, $h_\alpha(\mu)$ is finite for any $\mu$, and so asymptotic R\'enyi entropies might provide a tool to study random walks with heavy tails. We leave this question for future study.

\section{Preliminaries}
\subsection{R\'enyi entropy}
Let $\nu$ be a finitely supported probability distribution on a countable set $\Omega$. For $\alpha \in (0,1) \cup (1,\infty)$, the {\em $\alpha$-R\'enyi entropy} of $\nu$ is given by
\begin{align*}
    H_\alpha(\nu) = \frac{1}{1-\alpha}\log\sum_{\omega \in \Omega}\nu(\omega)^\alpha.
\end{align*}
For $\alpha \in \{0,1,\infty\}$ it is given by
\begin{align*}
    H_0(\nu) &= \log|\{\omega \,:\, \nu(\omega)>0\}|\\
    H_1(\nu) &= \sum_{\omega \in \Omega}\nu(\omega)\log\frac{1}{\nu(\omega)}\\
    H_\infty(\nu) &= \min_{\omega \in \Omega} \log\frac{1}{\nu(\omega)}.
\end{align*}
Hence $H_0$ is the logarithm of the size of the support, $H_1$ is the Shannon entropy, and $H_\infty$  is minus the log of the mass of the largest atom. Under this definition, it is well known (see, e.g., \cite[pp.\ 50--51]{principe2010information}) or immediate that
\begin{enumerate}[(i)]
    \item The map $\alpha \mapsto H_\alpha(\nu)$ is continuous and (weakly) decreasing. If $\nu$ is not the uniform distribution on a subset of $\Omega$ then it is strictly decreasing.
    \item For every $\alpha \in [0,\infty]$ it holds that $H_\alpha(\nu_1 \times \nu_2) = H_\alpha(\nu_1) + H_\alpha(\nu_2)$.
    \item For every map $f \colon \Omega \to \Omega'$ and every $\alpha \in [0,\infty]$ it holds that $H_\alpha(f_*\nu) \leq H_\alpha(\nu)$.
\end{enumerate}

Define the \emph{log-likelihood} random variable $L \colon \Omega \to \R$ by $L(\omega) = -\log \nu(\omega)$. Let $K_\nu \colon \R \to \R$ be the  cumulant generating function of $L$. That is, let
\begin{align*}
    K_\nu(t) = \log\E{\ee^{tL}} = \log\sum_{\omega}\ee^{tL(\omega)}\nu(\omega).
\end{align*}
Then for $\alpha \in (0,\infty)$,
\begin{align}
  \label{eq:K-H}
  K_\nu(1-\alpha) = (1-\alpha)H_\alpha(\nu).
\end{align} 
Note that this implies that $(1-\alpha)H_\alpha(\nu)$ is convex. As we show next, $H_\alpha(\nu)$ admits another re-parameterization which makes it convex for $\alpha \in (0,1)$. A similar version for $\alpha\in (1,\infty)$ has been observed in \cite{BuryakMishura2021convex}.
\begin{proposition}
\label{prop:convex}
For any finitely supported measure $\nu$,
$\beta \mapsto H_{1+\frac{1}{\beta}}(\nu)$ is convex for $\beta \in (-\infty,-1)$.
\end{proposition}
\begin{proof}
Let $\nu$ be a finitely supported probability measure on a set $\Omega$. For $f\colon \Omega \to \R$ and $p<0$, we denote $\ell^p$ norms by $\left\Vert f\right\Vert
_{p}^{p}:=\sum_\omega \left\vert f(\omega)\right\vert ^{p}\nu(\omega) $, where the sum is taken over all $\omega$ such that $f(\omega) \neq 0$. 

Fixing $p_{1},p_{2} < 0$ and $\theta \in (0,1)$, we let $p<0$ be such that
\[
\frac{1}{p}=\frac{\theta }{p_{1}}+\frac{1-\theta }{p_{2}},
\]
so that 
$$\frac{\theta p}{p_{1}}+\frac{(1-\theta )p}{p_{2}}=1.$$ 
Then, 
\begin{align*}
\left\Vert f\right\Vert _{p}^{p} =\sum_\omega \left\vert f(\omega)\right\vert ^{\theta
p}\left\vert f(\omega)\right\vert ^{(1-\theta )p}\nu(\omega)  
\leq \left\Vert f^{\theta p}\right\Vert _{\frac{p_{1}}{\theta p}%
}\left\Vert f^{(1-\theta )p}\right\Vert _{\frac{p_{2}}{(1-\theta )p}} 
=\left\Vert f\right\Vert _{p_{1}}^{\theta p}\left\Vert f\right\Vert
_{p_{2}}^{(1-\theta )p},
\end{align*}%
where the inequality follows from an application of the standard H\"older's inequality.

It then follows that  
\[
p\log \left\Vert f\right\Vert _{p}\leq \theta p\log \left\Vert f\right\Vert
_{p_{1}}+(1-\theta )p\log \left\Vert f\right\Vert _{p_{2}},
\]%
and therefore
\[
\log \left\Vert f\right\Vert _{p}\geq \theta \log \left\Vert f\right\Vert
_{p_{1}}+(1-\theta )\log \left\Vert f\right\Vert _{p_{2}}.
\]%
Since $\frac{1}{p}=\frac{\theta }{p_{1}}+\frac{1-\theta }{p_{2}}$, this shows that $p \mapsto \log \left\Vert f\right\Vert _{\frac{1}{p}}$ is
concave for $p<0$. 

Now let $f \colon \Omega \to \R$ be given by $f(\omega) = \nu(\omega)$. Then
\[
\log \left\Vert f\right\Vert _{\frac{1}{p}}=p\log \sum_{\omega}\nu(\omega)^{1+\frac{1}{p}}.
\]%
Thus, 
\[
H_{1+\frac{1}{\beta }}(\nu)=-\beta \log \sum_{\omega}\nu(\omega)^{1+\frac{%
1}{\beta }}=-\log \left\Vert f\right\Vert _{\frac{1}{\beta }}
\]%
is a convex function for $\beta \in (-\infty ,-1)$.
\end{proof}

We end this section with another simple observation.
\begin{lemma}
\label{clm:alpha_infty_bound}
For $\alpha>1$ it holds that 
\begin{align*}
    H_\alpha(\nu) 
    \leq \frac{\alpha}{\alpha-1}H_\infty(\nu).
\end{align*}
\end{lemma}
\begin{proof}
Since $\alpha>1$, we note that the factor $1-\alpha$ is negative, and hence
\begin{align*}
    H_\alpha(\nu) 
    = \frac{1}{1-\alpha}\log\sum_\omega \nu(\omega)^\alpha
    \leq \frac{1}{1-\alpha}\log\max_\omega \nu(\omega)^\alpha,
\end{align*}
which, by the definition of $H_\infty$, is equal to $\frac{\alpha}{\alpha-1}H_\infty(\nu)$.

\end{proof}
\subsection{Random walks on groups}
Let $G$ be a finitely generated discrete group, and let $\mu$ be a finitely supported probability measure on $G$. Denote convolution by $*$, and the $n$-fold convolution of $\mu$ with itself by $\mu^{(n)}$. We say that $\mu$ is symmetric if $\mu(g)=\mu(g^{-1})$. We say that $\mu$ is non-degenerate if the support of $\mu$ generates $G$ as a semi-group; equivalently, the $\mu$-random walk on $G$ is an irreducible Markov chain.

For $\alpha \in [0,\infty]$, define
\begin{align*}
  h_\alpha(\mu) = \lim_n\frac{1}{n}H_\alpha(\mu^{(n)}).
\end{align*}
We refer to the family $(h_\alpha)_\alpha$ as invariants, since if $\pi \colon G \to H$ is a group isomorphism that maps the probability measure $\mu$ on $G$ to the probability measure $\nu$ on $H$, then $h_\alpha(\mu) = h_\alpha(\nu)$.

It follows from property (ii) of $H_\alpha$ that $H_\alpha(\mu^{(n)} \times \mu^{(m)}) = H_\alpha(\mu^{(n)}) + H_\alpha(\mu^{(m)})$. And since $\mu^{(n+m)} = \mu^{(n)} * \mu^{(m)}$, it follows from property (iii) of $H_\alpha$ that $H_\alpha(\mu^{(n+m)}) \leq H_\alpha(\mu^{(n)}) + H_\alpha(\mu^{(m)})$. Thus
the map $n \mapsto H_\alpha(\mu^{(n)})$ is subadditive, and so the limit above exists and is finite. Moreover, this limit is  equal to the infimum of $\frac{1}{n}H_\alpha(\mu^{(n)})$. Thus, for a given $\mu$, $h_\alpha(\mu)$ is upper semi-continuous.

By definition, 
\begin{align*}
    h_1(\mu) = \lim_{n}\frac{1}{n}\sum_g\mu^{(n)}(g)\log\frac{1}{\mu^{(n)}(g)}
\end{align*}
is the random walk entropy, or Avez entropy. Similarly, $h_0(\mu)$ is the exponential growth rate of the group generated by the support of $\mu$:
\begin{align*}
    h_0(\mu) = \lim_n \frac{1}{n} H_0(\mu^{(n)}) = \lim_n \frac{1}{n}\log |\{g \,:\, \mu^{(n)}(g) > 0\}| = \lim_n \frac{1}{n}\log|B_\mu(n)|.
\end{align*}
Here $B_\mu(n)$ is the ball of radius $n$ with respect to the word metric defined by the generating set given by the support of $\mu$. 

The proof of the next claim is standard; see, e.g., \cite[Exercise 4.4.5]{lalley2023random}, where it is shown that $H_\infty(\mu^{(2n)})=-\log\mu^{(2n)}(e)$.
\begin{claim}
  \label{clm:spectral}
  Let $\mu$ be a finitely supported symmetric measure. Then $h_\infty(\mu)$ is the logarithm of the inverse of the spectral radius:
  \begin{align*}
      h_\infty(\mu) = \lim_{n \to \infty}\frac{1}{2n}\log\frac{1}{\mu^{(2n)}(e)}.
  \end{align*}
\end{claim}
There is an immediate corollary.

\begin{corollary}
\label{claim:amenable}
Let $G$ be a finitely generated amenable group, and let $\mu$ be a finitely supported, symmetric, non-degenerate probability measure on $G$. Then $h_\alpha(\mu) = 0$ for all $\alpha>1$.
\end{corollary}
\begin{proof}
By Kesten's Theorem~\cite{kesten1959symmetric,kesten1959full}, the assumptions on $\mu$ and the amenability of $G$ imply that the spectral radius of the random walk is $1$. Hence, by Claim~\ref{clm:spectral}, $h_\infty(\mu)=0$.

By Lemma~\ref{clm:alpha_infty_bound}, if $\alpha>1$ then  $H_\alpha \leq \frac{\alpha}{\alpha-1}H_\infty$ and so we have that $h_\alpha(\mu) \leq \frac{\alpha}{\alpha-1}h_\infty(\mu)=0$.
\end{proof}


\section{Proofs}

\begin{proof}[Proof of Theorem~\ref{thm:main-cont}]
Recall from \eqref{eq:K-H} that for $\alpha\in (0,\infty)$, it holds that
\begin{align*}
(1-\alpha)H_\alpha(\nu) = K_\nu(1-\alpha).     
\end{align*}
The cumulant generating function $K_\nu(t)$ of $L(\omega):=-\log \nu(\omega)$ is a convex function for $t\in \mathbb{R}$. For a fixed $t$, it is straightforward to check that the following hold:
\begin{enumerate}
    \item Additivity: $K_{\nu\times \nu'}(t)=K_{\nu}(t)+K_{\nu'}(t)$.
    \item Monotonicity under push-forwards: $K_{f_\ast\nu}(t)\leq K_\nu(t)$.
\end{enumerate}
By Fekete's lemma, the sequence of functions $t \mapsto \frac{1}{n}K_{\mu^{(n)}}(t)$ converges pointwise. The limit, which we denote by $k_\mu(t)$  is a convex function as well, and thus continuous. For $\alpha\geq 0$, we have $k_\mu(1-\alpha)=(1-\alpha) h_\alpha(\mu)$. Hence the restriction of $h_\alpha(\mu)$ to $[0,\infty) \setminus \{1\}$ is continuous. The function $\alpha\mapsto h_\alpha(\mu)=\inf_n \frac{1}{n}H_\alpha(\mu^{(n)}) $ is upper semi-continuous as an infimum of continuous functions. Since $H_\alpha(\nu)$ is decreasing for any $\nu$, so is $h_\alpha(\mu)$. Upper semi-continuity and monotonicity imply that $h_\alpha(\mu)$ is c\`agl\`ad. Thus we can conclude that it is left continuous at $1$. 

To prove the continuity at $\infty$, we note that $H_\alpha(\nu)$ is decreasing so that $H_\infty(\nu)=\inf_{\alpha\in (0,\infty)} H_\alpha(\nu)$. Therefore, 
\[h_\infty(\mu)=\inf_n \frac{1}{n}H_\infty(\mu^{(n)})=
\inf_n \inf_{\alpha\in (0,\infty)}\frac{1}{n}H_\alpha(\mu^{(n)})=\inf_{\alpha\in (0,\infty)} h_\alpha(\mu).\]
As $h_\alpha(\mu)$ is decreasing in $\alpha$, we obtain continuity at $\infty$.

We next show that $h_\alpha(\mu)$ is strictly decreasing on $[0,1]$ whenever it is larger than $h_1(\mu)$. By Proposition~\ref{prop:convex}, $H_{1+\frac{1}{\beta }}(\mu ^{(n)})$
is a convex function for $\beta \in (-\infty ,-1)$. Passing to the limit, 
\[
h_{1+\frac{1}{\beta }}(\mu )=\lim_{n}\frac{1}{n}H_{1+\frac{1}{\beta }}(\mu
^{(n)})
\]
is also a convex function for $\beta \in (-\infty ,-1)$. 

Denote $\bar h(\beta) = h_{1+\frac{1}{\beta}}(\mu)$, so that $\bar h$ is convex for $\beta \in (-\infty,-1)$. Since the map $\beta \mapsto 1+1/\beta$ is strictly decreasing and $h_\alpha(\mu)$ is decreasing, $\bar h$ is increasing. Since it is convex, it must be strictly increasing whenever it is not equal to its infimum. Hence, $h_\alpha(\mu) = \bar h(1/(\alpha-1))$ is strictly decreasing in $(0,1)$ whenever it is not equal to its infimum, $h_1(\mu)$.

We have so far established the first and second part of the claim. For the third, suppose that $\mu$ is symmetric, and let $X_1,X_2,\ldots$ be i.i.d.\ random variables taking value in $G$ with law $\mu$. Let $Z_n = X_1 \cdots X_n$ be the $\mu$-random walk, so that the law of $Z_n$ is $\mu^{(n)}$. Let $Z_1',Z_2',\ldots$ be an additional, independent $\mu$-random walk. Then, since the random walk is symmetric,
\begin{align*}
    h_2(\mu) 
    =-\lim_n \frac{1}{n}\log\sum_{g}\mu^{(n)}(g)^2=-\lim_n \frac{1}{n}\log \Pr{Z_n = Z_n'} = -\lim_n \frac{1}{n}\log\mu^{(2n)}(e).
\end{align*}
Thus
\begin{align*}
    h_2(\mu) =-\lim_n \frac{1}{n}\log \mu^{(2n)}(e)=  -2\lim_{n} \frac{1}{2n}\log \mu^{(2n)}(e) = 2h_\infty(\mu),
\end{align*}
because for symmetric random walks the maximum probability is achieved at the identity. This shows the claim for $\alpha=2$.

Comparing this to the definition of $h_2(\mu)$, we see that the entire sum over $g \in G$ is dominated by $g=e$, 
\[
h_2(\mu) = -\lim_n \frac{1}{2n}\log\sum_{g }\left(\mu^{(2n)}(g)\right)^2 = -\lim_n\frac{1}{2n}\log \left(\mu^{(2n)}(e)\right)^2.\]
The same holds for $\alpha >2$. We note that 
\begin{align*}
\lim_{n}\frac{1}{2n}\log \mu ^{(2n)}(e)^{\alpha } 
&=\lim_{n}\frac{1}{2n}\log \mu ^{(2n)}(e)^{2}+\lim_{n}\frac{1}{2n}\log\mu
^{(2n)}(e)^{\alpha -2} \\
&= \lim_{n}\frac{1}{2n}\log \sum_{g}\mu ^{(2n)}(g)^{2}+\lim_{n}\frac{1}{%
2n}\log\mu^{(2n)}(e)^{\alpha -2} \\
&=\lim_{n}\frac{1}{2n}\log \sum_{g}\mu ^{(2n)}(g)^{2}\mu ^{(2n)}(e)^{\alpha
-2} \\
&\geq \lim_{n}\frac{1}{2n}\log \sum_{g}\mu ^{(2n)}(g)^{\alpha } \\
&\geq \lim_{n}\frac{1}{2n}\log \mu ^{(2n)}(e)^{\alpha }
\end{align*}
where the second equality follows from the $\alpha=2$ case, and the  second last inequality is because the maximum probability
is achieved at the identity.  Therefore, all the inequalities above are in fact equalities and 
\begin{equation*}
\alpha h_\infty(\mu) = -\lim_{n}\frac{1}{2n}\log \mu ^{(2n)}(e)^{\alpha }=-\lim_{n}\frac{1}{2n}\log
\sum_{g}\mu ^{(2n)}(g)^{\alpha } = (\alpha-1)h_\alpha(\mu).
\end{equation*}%

\end{proof}




\subsection{The free group}
Let $G = \mathbb{F}_d$ be the free group with $d \geq 2$ generators and let $\mu$ be the uniform distribution on the set of $d$ generators and their inverses. As above, let $X_1,X_2,\ldots$ be i.i.d.\ random variables taking value in $G$ with law $\mu$. Let $Z_n = X_1 \cdots X_n$ be the $\mu$-random walk, and let $D_n = |Z_n|$ be the distance between $Z_n$ and the origin. Then $D_1,D_2,\ldots$ is a Markov chain, where $\Pr{D_n=1}{D_{n-1}=0}=1$, and for $r > 0$, $\Pr{D_n=r-1}{D_{n-1}=r}=\frac{1}{2d}$ and $\Pr{D_n=r+1}{D_{n-1}=r}=1-\frac{1}{2d}$.

By the symmetry of the random walk, if $|g|=|h|$ then $\mu^{(n)}(g)=%
\mu^{(n)}(h)$. The number of elements $g \in \mathbb{F}_d$ with $|g|=k$ is
equal to $2d\cdot (2d-1)^{k-1}$. Hence
\begin{equation}  \label{Zn}
\mu^{(n)}(g) = \Pr{Z_n=g} = \frac{\Pr{D_n=k}}{2d\cdot (2d-1)^{k-1}}.
\end{equation}
It follows that when $\alpha\neq 1$,
\begin{align*}
  H_{\alpha }(\mu ^{(n)})
  &=\frac{1}{1-\alpha }\log \sum_{g}{\Pr {Z_{n}=g}^{\alpha }} \\
  &=\frac{1}{1-\alpha }\log \sum_{k=0}^{n}\sum_{|g|=k}{\left( \frac{\Pr {
D_{n}=k}}{2d\cdot (2d-1)^{k-1}}\right) ^{\alpha }} \\
  &=\frac{1}{1-\alpha }\log \sum_{k=0}^{n}\Pr{D_{n}=k}^{\alpha
}(2d\cdot (2d-1)^{k-1})^{1-\alpha }\;,
\end{align*}
and so
\begin{align*}
  h_\alpha(\mu)
  =\lim_n \frac{1}{n}\frac{1}{1-\alpha }\log \sum_{k=0}^{n}\Pr{D_{n}=k}^{\alpha
}(2d\cdot (2d-1)^{k-1})^{1-\alpha }.
\end{align*}
And since the sum is bounded below by its maximum and above by $n+1$ times its maximum,
\begin{equation}
\label{Freeh}
\begin{split}
  h_\alpha(\mu)
  &=\lim_n \frac{1}{n}\frac{1}{1-\alpha }\log \max_{k \in \{0,\ldots,n\}} \Pr{D_{n}=k}^{\alpha
}(2d-1)^{(1-\alpha )k}\\
  &=\lim_n \frac{1}{n}\frac{1}{1-\alpha }\max_{k \in \{0,\ldots,n\}} \alpha \log \Pr{D_{n}=k} + (1-\alpha )k\log (2d-1)\\
  &=\lim_n \frac{1}{n}\frac{1}{1-\alpha }\max_{k \in \{\lceil n/2 \rceil,\ldots,n\}} \alpha \log \Pr{D_{n}=2k-n} + (1-\alpha )(2k-n)\log (2d-1),
  \end{split}
\end{equation}
where the last equality is just a change of variables.

Let $E_1,E_2,\ldots$ be a random walk on $\Z$ with $\Pr{E_{n+1} = E_n+1} = 1-\frac{1}{2d}$ and $\Pr{E_{n+1} = E_n-1} = \frac{1}{2d}$.  The only difference between $E_n$ and $D_n$ is that $D_n$ is reflected at $0$ while $E_n$ is allowed to travel to the left of $0$. We then try to understand the probability distribution of $D_n$ from that of $E_n$.
\begin{lemma}
    \label{EnDn} For any $k\geq 0$ and any $n$, we have 
\[\frac{1}{n}\mathbb{P}[E_{n}=k]\leq\mathbb{P}[D_{n}=k]\leq 2\mathbb{P}[E_{n}=k].\]
\end{lemma}
\begin{proof}
We begin by considering the case where $k=0$. In particular, we will prove
that for any $n$, 
\begin{equation*}
\frac{1}{n}\mathbb{P}[E_{n}=0]\leq \mathbb{P}[D_{n}=0]\leq \frac{2d-1}{2d-2}%
\mathbb{P}[E_{n}=0].
\end{equation*}%
For this argument, we may assume that $n$ is even (since, if $n$ were odd, then both sides of the equation above would equal zero). \newline
We define a coupling of $D_{i}$ and $E_{i}$ by first sampling $E_{i}$ and
then setting 
\begin{eqnarray*}
D_{i+1}-D_{i} &=&E_{i+1}-E_{i}\quad \text{ if }D_{i}\geq 1 \\
D_{i+1} &=&D_{i}+1\quad \quad \;\text{ if }D_{i}=0.
\end{eqnarray*}%
From above, we see that the random walks are synchronous but their positions
may be shifted upon returning to $0$, which gives $E_{n}\leq D_{n}$. Hence, 
\begin{equation*}
\mathbb{P}[D_{n}=0]\leq \mathbb{P}[E_{n}\leq 0].
\end{equation*}

Then for $n=2m$ even, we have 
\begin{eqnarray*}
\mathbb{P}[E_{2m}\leq 0] &=&\sum_{i=0}^{m}\mathbb{P}[E_{2m}=-2i] \\
&=&\sum_{i=0}^{m}\left( \frac{2d-1}{2d}\right) ^{m-i}\left( \frac{1}{2d}%
\right) ^{m+i}\binom{2m}{m-i} \\
&\leq &\sum_{i=0}^{m}\left( \frac{2d-1}{2d}\right) ^{m-i}\left( \frac{1}{2d}%
\right) ^{m+i}\binom{2m}{m} \\
&\leq &\frac{1}{1-\frac{1}{2d-1}}\left( \frac{2d-1}{2d}\right) ^{m}\left( 
\frac{1}{2d}\right) ^{m}\binom{2m}{m} \\
&=&\frac{2d-1}{2d-2}\mathbb{P}[E_{2m}=0]
\end{eqnarray*}%
Thus, we have shown that when $k=0$, 
\begin{equation*}
\mathbb{P}[D_{n}=0]\leq \frac{2d-1}{2d-2}\mathbb{P}[E_{n}=0].
\end{equation*}

For the lower bound, we first note that if $E_{n}$ remains nonnegative then $%
D_{n}=E_{n}$,
which follows from the coupling defined above. Therefore, 
\begin{equation*}
\left\{ E_{2m}=0\text{ and }E_{i}\geq 0\text{ for }i<2m\right\} \subset
\{D_{2m}=0\}
\end{equation*}%
and so $\mathbb{P}[E_{2m}=0$ and $E_{i}\geq 0$ for $i<2m]$ $\leq \mathbb{P}%
[D_{2m}=0]$. To show the  desired lower bound, note that by the Bertrand's Ballot Theorem, 
\begin{eqnarray*}
\mathbb{P}[D_{2m} =0]\geq \mathbb{P}[E_{2m}=0\text{ and }E_{i}\geq 0\text{
for }i<2m] 
=\frac{1}{m+1}\mathbb{P}[E_{2m}=0].
\end{eqnarray*}

For $k>0$, note that by again using the Bertrand's Ballot Theorem we have that
\begin{align*}
\mathbb{P}[D_{n} =k]\geq \mathbb{P}[E_{n}=k\text{ and }E_{i}> 0\text{
for }i<n] = \frac{k}{n}\mathbb{P}[E_{n}=k] \geq \frac{1}{n}\mathbb{P}[E_{n}=k].
\end{align*}
This shows the lower bound. For the upper bound, we will use induction on $%
n$ to prove that for any $k$
\begin{equation*}
\mathbb{P}[D_{n}=k]\leq 2\mathbb{P}[E_{n}=k].
\end{equation*}%
The base case $n=0$ is trivial. For the
inductive step, we suppose that the above inequality holds for some $n$ and
all $k$. For $k>1$, we have%
\begin{eqnarray*}
\mathbb{P}[D_{n+1}=k] &=&\frac{2d-1}{2d}\mathbb{P}[D_{n}=k-1]+\frac{1}{2d}%
\mathbb{P}[D_{n}=k+1] \\
&\leq &2\left( \frac{2d-1}{2d}\mathbb{P}[E_{n}=k-1]+\frac{1}{2d}\mathbb{P}%
[E_{n}=k+1]\right)  \\
&=&2\mathbb{P}[E_{n+1}=k].
\end{eqnarray*}%

For $k=1$, we have 
\begin{eqnarray*}
\mathbb{P}[D_{n+1}=1] &=&\mathbb{P}[D_{n}=0]+\frac{1}{2d}\mathbb{P}[D_{n}=2]
\\
&\leq &\frac{2d-1}{2d-2}\mathbb{P}[E_{n}=0]+\frac{1}{2d}\cdot 2\mathbb{P}%
[E_{n}=2] \\
&\leq &2\left( \frac{2d-1}{2d}\mathbb{P}[E_{n}=0]+\frac{1}{2d}\mathbb{P}%
[E_{n}=2]\right)  \\
&=&2\mathbb{P}[E_{n+1}=1].
\end{eqnarray*}%
This completes the induction step.
\end{proof}

The new random walk $E_n$ may be understood very well. First note that,
\begin{align*}
  \Pr{E_n=2k-n}
  &= \binom{n}{k}\left(\frac{2d-1}{2d}\right)^k\left(\frac{1}{2d}\right)^{n-k}.
\end{align*}
By a standard estimate, 
\begin{align*}
 \frac{1}{n+1}\ee^{nH(k/n)} \leq \binom{n}{k} \leq \ee^{nH(k/n)},   
\end{align*}
where $H(p) = -p\log p-(1-p)\log(1-p)$, and so
\begin{align*}
  \log \Pr{E_n=2k-n}
  = -n\log (2d)+n H(k/n) +k\log (2d-1)+o(n).
\end{align*}
Furthermore, by Lemma~\ref{EnDn}, for $k \geq 0$, $\log \Pr{D_n=k} = \log \Pr{E_n=k} + o(n)$. Hence
\begin{align*}
  \frac{1}{n}\log \Pr{D_n=2k-n} = -\log (2d)+ H(k/n) +k/n\log (2d-1)+o(1)
\end{align*}
for $k \geq n/2$. Inserting this into (\ref{Freeh}) yields
\begin{align*}
  h_\alpha(\mu)
  =\lim_n \frac{1}{1-\alpha }\max_{k \in \{\lceil n/2 \rceil,\ldots,n\}} &\alpha (-\log (2d)+ H(k/n) +k/n\log (2d-1))\\
  &\quad+ (1-\alpha )(2k/n-1)\log (2d-1).
\end{align*}
The expression being maximized is a function of $k/n$, which we denote by
$f_\alpha \colon [1/2,1] \to \R$:
\begin{align*}
    f_\alpha(p) = \alpha (-\log (2d)+ H(p) +p\log (2d-1))+ (1-\alpha )(2p-1)\log (2d-1).
\end{align*}

\vspace{2mm} Using this construction, we may now prove Theorem~\ref{thm:main-free}.

\begin{proof}[Proof of Theorem~\ref{thm:main-free}]
For $\alpha\neq 1$, from the notation established above, we have
\begin{align*}
    h_\alpha(\mu)
  =\frac{1}{1-\alpha }\lim_n \max_{k \in \{\lceil n/2 \rceil,\ldots,n\}}f_\alpha(k/n).
\end{align*}
Since $f_\alpha$ is continuous,
\begin{align*}
h_\alpha(\mu) = \frac{1}{1-\alpha }\max_{p \in [1/2,1]}f_\alpha(p).
\end{align*}
Since $f_\alpha$ is furthermore differentiable and concave, this maximum is achieved either at the end points of the interval, or else at the unique $p_\alpha^*$ at which the derivative of $f_\alpha$ vanishes. A simple calculation shows that this is given by
\begin{align*}
    p_\alpha^* = \frac{(2d-1)^{2/\alpha}}{(2d-1)+(2d-1)^{2/\alpha}}
\end{align*}
for $\alpha \in (0,2)\backslash\{1\}$. For these $\alpha$, we have
\begin{align}\label{eq:free_entropy}
  h_\alpha(\mu) &= \frac{1}{1-\alpha }f_\alpha(p_\alpha^*),
\end{align}
where we recall that
\begin{align*}
    f_\alpha(p) = \alpha (-\log (2d)+ H(p) +p\log (2d-1))+ (1-\alpha )(2p-1)\log (2d-1).
\end{align*}

This can be written out as an elementary (but unwieldy) function of $\alpha$, which we omit. Importantly, this implies that $h_\alpha(\mu)$ is an analytic function of $\alpha$ on $(0,2)\backslash\{1\}$. Moreover, $f_\alpha(p_\alpha^*)$, viewed as a complex function, is holomorphic around a neighborhood of $\alpha=1$. Some computations show that $f_\alpha(p_\alpha^*)$ has a zero of order 1 at $\alpha=1$ so $h_\alpha(\mu) = \frac{1}{1-\alpha }f_\alpha(p_\alpha^*)$ is a meromorphic function whose singularity at $\alpha=1$ can be removed. In other words,  $h_\alpha(\mu)$ is real analytic at $\alpha=1$. 

For $\alpha\geq2$, the R\'enyi  entropy enters a difference phase $h_\alpha(\mu)=\frac{\alpha}{\alpha-1}h_\infty(\mu)$. Comparing this with (\ref{eq:free_entropy}), we see that $h_\alpha(\mu)$ is first differentiable, but not second differentiable, at $\alpha=2$.

\end{proof}

\subsection{The lamplighter group}

Let $G = L \wr \Z$ be the lamplighter group on $\Z$ with lamps in some non-trivial finite group $L$. We denote an element of $G$ by $g = (f,z)$ where $f \colon \Z \to L$ has finite support, and $z \in \Z$. Denote by $\pi:G\rightarrow \mathbb{Z}$ be the group homomorphism $\pi(f,z):=z$ and denote by $\tau(f,z):=f$ the projection to the lamp configurations. Let $\mu$ be a symmetric, finitely supported measure on $G$.


We prove Theorem~\ref{thm:main-ll-positivity} by analyzing ``tilted'' versions of the $\mu$-random walk. For $t \geq 0$ define the tilted measure $\mu_t$ by
\begin{align*}
  \mu_t(g) = C_t^{-1} \mu(g)\ee^{t\pi(g)}
\end{align*}
where
\begin{align*}
    C_t = \sum_{g' \in G} \mu(g')\ee^{t\pi(g')}
\end{align*}
is the moment generating function of $\pi_*\mu$, evaluated at $t$. Note that tilting commutes with convolution, so that $\mu_t^{(k)}$ is well-defined.

Note that $C_t'(0) = \sum_h \mu(h) \pi(h)=0$ because $\mu$ is symmetric, and hence the expectation of $\pi_*\mu$ is equal to zero. The second derivative of $C_t$ at zero is the variance $v >0$ of $\pi_*\mu$, and hence $C_t = 1+v t^2 + o(t^2)$.

Let $Z_1,Z_2,\ldots$ be a $\mu_t$-random walk, and let $Q^t_n = \{ \pi(Z_1),\ldots,\pi(Z_n)\}$ be the set of locations visited by the lamplighter. 
\begin{lemma}
\label{lemma:visited}
    There is a constant $q >0$ such that, for all $t>0$ small enough,  $\liminf_n \frac{1}{n}|Q^t_n| \geq q t$ almost surely.
\end{lemma}
\begin{proof}
    The tilted random walk has a positive drift of%
\begin{eqnarray*}
\Delta_{t} &:=&C_{t}^{-1}\sum_{g}\mu (g)\exp (t\pi (g))\pi (g) \\
&=&C_{t}^{-1}\left( \sum_{g:\pi (g)<0}\mu (g)\exp (t\pi (g))\pi
(g)+\sum_{\pi (g)>0}\mu (g)\exp (t\pi (g))\pi (g)\right)  \\
&\geq &C_{t}^{-1}\left( \sum_{g:\pi (g)<0}\mu (g)\pi (g)+\sum_{\pi (g)>0}\mu
(g)(1+t\pi (g))\pi (g)\right)  \\
&\geq&C_{t}^{-1}t\sum_{\pi (g)>0}\mu (g)(\pi (g))^{2}.
\end{eqnarray*}%
As $C_{t}=1+o(t)$, we have 
\[
\Delta_{t}>\beta t
\]%
for $t>0$ small enough, where we set $\beta :=\frac{1}{2}\sum_{\pi (g)>0}\mu
(g)(\pi (g))^{2}$. Note that $\beta >0$.

Since $\mu $ has finite support, there exists some $M>0$ such that $%
\left\vert \pi (g)\right\vert <M$ for any $g$ in the support of $\mu $. Then
the step size under $\mu _{t}$ is also bounded by $M$. Let $%
Z_{1},Z_{2},\ldots $ be a $\mu _{t}$-random walk. For any $\varepsilon >0$,
by law of large numbers, we have%
\[
\lim_{n}\mathbb{P}(\pi (Z_{n})>(\Delta_{t}-\varepsilon )n)=1.
\]%
Note that in order for the lamplighter to arrive at some position $N$ at
time $n$, the lamplighter must have visited at least $N/M$ many positions on
the way. Thus. 
\[
\left\vert Q_{n}^{t}\right\vert \geq \frac{\pi (Z_{n})}{M}.
\]%
Therefore, for $t>0$ small enough,  
\[
\liminf_{n}\frac{1}{n}\left\vert Q_{n}^{t}\right\vert \geq \frac{%
\Delta_{t}-\varepsilon }{M}>\frac{\beta t-\varepsilon }{M}.
\]%
Set $q:=\frac{\beta }{M}$. As $\varepsilon $ is arbitrary, we have 
\[
\liminf_{n}\frac{1}{n}\left\vert Q_{n}^{t}\right\vert \geq qt.
\]
\end{proof}

For the tilted random walk, the lamp status at the origin is a non-trivial tail random variable so the random walk has positive asymptotic Shannon entropy: $h_1(\mu_t)>0$. The next proposition shows that this entropy grows at least linearly with $t$, for small $t$. 
\begin{proposition}
\label{prop:ll}
  Suppose $\mu$ is non-degenerate. There is a constant $c >0$ such that $h_1(\mu_t) \geq c t$ for $t>0$ small enough.
\end{proposition}
\begin{proof}
  We first show that we can assume that $\mu(e)>0$. Let $\eta = \frac{1}{2}\mu + \frac{1}{2}\delta_e$, where $\delta_e$ is the point mass at the identity of $G$. Define the tilted measures $\eta_t$ similarly to the definition of $\mu_t$. Then
  \begin{align*}
      \eta_t(g)  = \frac{1}{\frac{1}{2}C_t+\frac{1}{2}}\left( \frac{1}{2}\mu(g)\ee^{t\pi(g)}+\frac{1}{2}\delta_e(g)\right) = \frac{1}{\frac{1}{2}C_t+\frac{1}{2}}\left(\frac{1}{2}C_t\mu_t(g)+\frac{1}{2}\delta_e(g)\right),
  \end{align*}
  and so
  \begin{align*}
    \eta_t = \alpha_t \mu_t + (1-\alpha_t)\delta_e
  \end{align*}
  for some $\alpha_t$ that tends continuously to $1/2$ as $t$ tends to zero.
  
  Since $h_1(\alpha \mu_t + (1-\alpha)\delta_e) = \alpha h_1(\mu_t)$, it follows that if $h_1(\eta_t) > ct$ then $h_1(\mu_t) > ct$ for all $t$ small enough. It hence suffices to prove the claim for $\eta$, so that it follows for $\mu$. We thus assume without loss of generality that $\mu(e) > 0$.

  Denote by $f^\ell \colon \mathbb{Z} \to L$ the function given $f^\ell(0)=\ell$ and $f^\ell(z)=e$ for all $z \neq 0$, and where, by slight notation overloading, $e$ is the identity of $L$.
  
  Since $\mu$ is non-degenerate,  and since $\mu(e)>0$,   there is some $k$ large enough and $\eps>0$ such that $\mu^{(k)}(f^\ell,0)>2\eps$ for all $\ell \in L$. We hence again assume without loss of generality that this already holds for $\mu$, since $h(\mu_t^{(k)}) = k h(\mu_t)$. 
  For all $t$ small enough we will have that $\mu_t(f^\ell,0)>\eps$ for all $\ell \in L$. Let $\nu$ be the uniform distribution over $\tilde L = \{(f^\ell,0)\,:\,\ell \in L\}$. Then we can write $\mu_t = \eps\nu +(1-\eps)\bar\mu_t$ where 
  \begin{align*}
      \bar\mu_t = \frac{\mu_t - \eps\nu}{1-\eps}
  \end{align*}
  is a probability measure.

  Let $\bar X_1,\bar X_2,\ldots$ be i.i.d.\ random variables with distribution $\bar \mu_t$. Let $W_n$ be i.i.d.\ random variables with distribution $\nu$.
   Let $B_1,B_2,\ldots$ be i.i.d.\ Bernoulli random variables with $\Pr{B_n=1}=\eps$.
  Let 
  \begin{align*}
      X_n = 
      \begin{cases}
        \bar X_n&\text{if } B_n=0\\
        W_n&\text{if } B_n=1. 
      \end{cases}
  \end{align*}
  Then $X_n$ has distribution $\mu_t$. Let $Z_n = X_1 \cdot X_2 \cdots X_n$ be a $\mu_t$-random walk on $G$.

  We would like to show that $H_1(Z_n) \geq n c t$ for some $c$, all $t$ small enough and all $n$ large enough. 

  Let 
  \begin{align*}
      P_n = \{ z \in Z\,:\, \exists k \leq n, \pi(Z_k)=z, B_k=1\}.
  \end{align*}
  This is the random set of locations at which the lamplighter was at times in which $B_k=1$. Since $\pi(W_n)=0$, the probability that $\pi(Z_{n}) \in P_n$ is at least $\eps$. Hence $\E{|P_n|} \geq \eps\E{|Q^t_n|}$, and by Lemma~\ref{lemma:visited}, 
  there is a constant $c$ such that for all $t$ small enough, $\E{|P_n|} \geq n c  t$ for all $n$ large enough. 
  
  Denote $F_n = \tau(Z_n)$ the lamp configuration at time $n$. Suppose for a moment that $\bar{X}_1,\ldots\bar{X}_n$ and $B_1,\ldots,B_n$ are given. As $\pi(W_k)=0$, we will know the location of the lamplighter at all times up to $n$ and so we can deduce $P_n$. We will also know the state of the lamps outside of $P_n$. On $P_n$, the conditional distribution of lamps is i.i.d.\ uniform, since $W_k$ are uniform. Hence, the conditional entropy of $F_n$ is exactly $|P_n|\log |L|$ and
  \begin{align*}
      H_1(Z_n) 
      &\geq H_1(Z_n|\bar{X}_1,\ldots\bar{X}_n,B_1,\ldots,B_n)\\
      &\geq H_1(F_n|\bar{X}_1,\ldots\bar{X}_n,B_1,\ldots,B_n)\\ 
      &= \E{|P_n|}\log |L|\\
      &\geq n c t.
  \end{align*}
\end{proof}

With the above construction, we are now in position to prove Theorem~\ref{thm:main-ll-positivity}.

\begin{proof}[Proof of Theorem~\ref{thm:main-ll-positivity}]

Suppose by contradiction, for some $0<\alpha ^{\prime }<1$ we have 
\begin{equation*}
h_{\alpha ^{\prime }}(\mu )=\lim \frac{1}{n}\frac{1}{1-\alpha ^{\prime }}\log \sum_{g}\mu^{(n)}(g)^{\alpha ^{\prime }}=0.
\end{equation*}
Equivalently, 
\begin{equation}
\sum_{g}\mu^{(n)}(g)^{\alpha ^{\prime }}=\exp (o(n)).
\label{tail2}
\end{equation}

Let $\alpha =\frac{1+\alpha ^{\prime }}{2}<1$. For a fixed $t>0$, let $A > 0$
be chosen later. Then by Cauchy-Schwarz inequality, 
\begin{align}
\left( \sum_{\pi (g)\geq An}\mu^{(n)}(g)^{\alpha }\exp (\alpha t\pi
(s))\right) ^{2}   
&\leq \left( \sum_{\pi (g)\geq An}\mu^{(n)}(g)^{\alpha ^{\prime }}\right)\left( \sum_{\pi (g)\geq An}\mu^{(n)}(g)\exp (2\alpha t\pi (g))\right) . \label{tail1}
\end{align}

For the first term in the multiplication, we know that 
\begin{equation*}
\sum_{\pi (g)\geq An}\mu^{(n)}(g)^{\alpha ^{\prime }}\leq \sum_{s}\mu^{(n)}
(g)^{\alpha ^{\prime }}=\exp (o(n)).
\end{equation*}

For the second term, we note that $\pi_{\ast }\mu $ is a zero-drift random
walk on $\mathbb{Z}$, since $\mu$ is symmetric. By the Hoeffding bound, there exists a constant $c>0$
only depending on $\pi_{\ast }\mu $ such that 
\begin{equation*}
\pi_\ast\mu^{(n)}(k) = \sum_{\pi (g)=k}\mu^{(n)}(g)\leq \sum_{\pi (g)\geq k}\mu^{(n)}(g)\leq \exp \left(-c%
\frac{k^{2}}{n}\right).  
\end{equation*}

Taking a weighted sum of the above inequality, we have 
\begin{align*}
\sum_{\pi (g)\geq An}\mu^{(n)}(g)\exp (2\alpha t\pi (g)) 
&=\sum_{k=\lceil An\rceil }^{\infty }\sum_{\pi (g)=k}\mu^{(n)}(g)\exp
(2\alpha tk) \\
&\leq \sum_{k=\lceil An\rceil }^{\infty }\exp (-c\frac{k^{2}}{n})\exp
(2\alpha tk) \\
&\leq \sum_{k=\lceil An\rceil }^{\infty }\exp (-ckA)\exp (2\alpha tk) \\
&\leq \frac{M}{cA-2\alpha t}\exp (An(-cA+2\alpha t))
\end{align*}
for some $M$ large enough, and 
given that $cA-2\alpha t>0$. Sending $n\rightarrow \infty $, the last line goes
to $0$ so  
\begin{equation}
\lim_{n\rightarrow \infty }\sum_{\pi (g)\geq An}\mu^{(n)}(g)\exp (2\alpha
t\pi (g))\rightarrow 0.  \label{tail3}
\end{equation}%
(\ref{tail1}) together with (\ref{tail2}) and (\ref{tail3}) yields 
\begin{equation}
\sum_{\pi (g)\geq An}\mu^{(n)}(g)^{\alpha }\exp (\alpha t\pi (g))=\exp (o(n)).
\label{tail4}
\end{equation}

For $t$ small enough, by Proposition \ref{prop:ll}, there exists a constant $c'>0$  such that
\begin{align*}
c't\leq h(\mu _{t})=\lim_{n}\frac{1}{n}\frac{1}{1-\alpha }\log
\sum_{g}C_{t}^{-n\alpha }\mu^{(n)}(g)^{\alpha }\exp (\alpha t\pi (g)).
\end{align*}
As $C_{t}=1+O(t^{2})$, we have
\begin{equation*}
ct+O(t^{2})\leq \lim_{n}\frac{1}{n}\frac{1}{1-\alpha }\log \sum_{g}\mu^{(n)}(g)^{\alpha }\exp (\alpha t\pi (g)).
\end{equation*}

Because the tail is exponentially small as in (\ref{tail4}), we get%
\begin{eqnarray*}
ct+O(t^{2}) &\leq &\lim_{n}\frac{1}{n}\frac{1}{1-\alpha }\log \sum_{\pi
(g)\leq An}\mu^{(n)}(g)^{\alpha }\exp (\alpha t\pi (g)) \\
&\leq &\lim \frac{1}{n}\frac{1}{1-\alpha }\log \sum_{\pi (g)\leq An}\mu^{(n)}(g)^{\alpha }\exp (\alpha tAn) \\
&\leq &\frac{\alpha tA}{1-\alpha }+\lim \frac{1}{n}\frac{1}{1-\alpha }\log
\sum_{\pi (g)\leq An}\mu^{(n)}(g)^{\alpha } \\
&\leq &\frac{\alpha tA}{1-\alpha }+\lim \frac{1}{n}\frac{1}{1-\alpha }\log
\sum_{g}\mu^{(n)}(g)^{\alpha }
\end{eqnarray*}%
We set $A=(4\alpha t)/c$ so that $cA-2\alpha t>0$. Then for $t$ small, 
\begin{equation*}
\frac{\alpha tA}{1-\alpha }=\frac{\alpha }{1-\alpha }\frac{4\alpha }{c}%
t^{2}<ct+O(t^{2}).
\end{equation*}%
Thus, 
\begin{equation*}
h_{\alpha }(\mu )=\lim \frac{1}{n}\frac{1}{1-\alpha }\log \sum_{g}\mu^{(n)}(g)^{\alpha }>0.
\end{equation*}%
However, it follows that $h_{\alpha'}(\mu)\geq h_\alpha(\mu)>0$
as $%
\alpha >\alpha ^{\prime }$, a contradiction to our assumption on $\alpha
^{\prime }$. This concludes our proof. 
\end{proof}

\begin{remark}
    The proof of Theorem \ref{thm:main-ll-positivity} should work more generally for other groups which have $\mathbb{Z}$ as a quotient group and $\mathbb{Z}$ acts interestingly on the group, such as one-dimensional Baumslag-Solitar groups. 
\end{remark}

We next consider a particular case of the lamplighter group where $L = \Z_2$. We calculate the R\'enyi entropy of the 
``switch-walk-switch'' (SWS) random walk on the lamplighter group $\mathbb{Z}_2\wr \mathbb{Z}$, i.e., the $\mu$-random walk where $\mu = \eta * \sigma * \eta$, $\eta(f^1,0)=\eta(0,0) = 1/2$, and $\sigma(0,-1)=\sigma(0,1)=1/2$.
Thus, at every step the
lamplighter switches the lamp at the current location with probability one half, takes a step of the simple random walk and then again flips the lamp at the current
location with probability one half.

\begin{theorem}
\label{thm:sws}
    Let $\mu$ be the SWS walk on $\Z_2 \wr \Z$. Then 

    \begin{align}
    \label{eq:sws}
        h_\alpha(\mu) = \varphi_\alpha\left(\frac{4^{1/\alpha}-4}{4^{1/\alpha}+4}\right) 
    \end{align}
    for all $\alpha \in (0,1]$, where
    \begin{equation}
    \varphi_\alpha(p) 
    = p \log 2 - \frac{\alpha}{2(1-\alpha)}\left[ (1 - p) \log(1-p) + 
    (1 + p) \log(1 + p)\right].
    \end{equation}
\end{theorem}
Note that Theorem~\ref{thm:main-ll} is an immediate consequence, since Theorem~\ref{thm:sws} shows that $h_\alpha(\mu)$ is an elementary function on $(0,1]$. 
\begin{proof}[Proof of Theorem~\ref{thm:sws}]
Fix $n$. We first express the asymptotic R\'enyi entropies in terms of the number of positions the lamplighter ever visits up to time $n$. 
Let $S_{m}=\pi (Z_{m})$ be the position of the lamplighter at time $m$. Note that $S_m$ is a simple random walk on $\mathbb{Z}$. Let $L:=\inf \left\{ S_{m},0\leq m \leq n\right\} $ and $R:=\sup \left\{ S_{m},0\leq m\leq n\right\} $. The number of positions the lamplighter ever
visits up to time $n$ is then $R-L+1$. As $L,R \in [-n,n]$, for each $k$ there exists $\ell_k,r_k$ with $r_k-\ell_k+1=k$ such that 
\begin{align}
\label{eq:LR}
\Pr{L=\ell_k,R=r_k}\geq \frac{1}{5n^{2}}\Pr{R-L+1=k}.
\end{align}

We slightly abuse notation and let $\mu^{(n)}(f):=\sum_x \mu^{(n)}(f,x)$  denote the probability that the lamp configuration at time $n$ is equal to $f$. 
Fix $\alpha \in (0,1)$. Then
\begin{align*}
\log \sum_{g}\mu ^{(n)}(g)^{\alpha } 
=\log \sum_{f}\sum_{x}\mu^{(n)}(f,x)^{\alpha } 
\geq \log \sum_{f}\mu ^{(n)}(f)^{\alpha },
\end{align*}
since $q^\alpha+p^\alpha \geq (p+q)^\alpha$ for all $p,q \in [0,1]$. For each $k$, we can further bound this from below by restricting the sum: 
\begin{align*}
\log \sum_{g}\mu ^{(n)}(g)^{\alpha } 
 \geq \log \sum_{\text{supp}(f)\subseteq [ r_k,\ell_k]}\mu^{(n)}(f)^{\alpha }.
\end{align*}
We note that conditional on $L$ and $R$, the lamp configuration follows a uniform distribution over $\{0,1\}^{[L,R]}\times \{0\}^{(-\infty ,\infty )\setminus [L,R]}$. 
Hence, for $\text{supp}(f)\subseteq [ \ell_k,r_k]$, 
\begin{align*}
    \mu^{(n)}(f)=\Pr{\tau(Z_n)=f} \geq \Pr{\tau(Z_n)=f, R=r_k, L=\ell_l} = 2^{-k}\Pr{R=r_k, L=\ell_k}.
\end{align*}
Thus, since the sum over all such $f$ has $2^k$ summands,
\begin{align}\label{eq:llineq}
\log \sum_{g}\mu ^{(n)}(g)^{\alpha } 
&\geq \max_{k}\log \left[\left( \Pr{L=\ell_k,R=r_k}\cdot 2^{-k}\right) ^{\alpha }2^{k}\right] \nonumber \\
&=\max_{k}\log \left[\left( \Pr{R-L+1=k} \cdot 2^{-k}\right) ^{\alpha}2^{k}\right]+o(n).
\end{align}

For the other direction,
\begin{align*}
\sum_{g}\mu ^{(n)}(g)^{\alpha } 
 = \sum_{f}\sum_{x \in [-n,n]}\mu ^{(n)}(f,x)^{\alpha } \leq \sum_{f}\sum_{x\in [-n,n]}\mu ^{(n)}(f)^{\alpha }=(2n+1)\sum_{f}\mu ^{(n)}(f)^{\alpha }
\end{align*}
since $\mu^{(n)}(f,x) \leq \mu^{(n)}(f)$. For a given configuration $f:\mathbb{Z}\rightarrow \mathbb{Z}/2\mathbb{Z}$, let $r(f)$ and $\ell(f)$ denote the rightmost and leftmost
position on which lamps are on. Then
\begin{align*}
    \mu^{(n)}(f) = \Pr{\tau(Z_n) = f} 
    &= \sum_{m,s}\Pr{R=m,L=s,\tau(Z_n)=f}\\
    &= \sum_{m\geq r(f),s\leq \ell(f)}\Pr{R=m,L=s}2^{-(m-s+1)}.
\end{align*}
Hence
\begin{align*}
\sum_{g}\mu ^{(n)}(g)^{\alpha } 
&\leq(2n+1)\sum_{f}\left( \sum_{m\geq r(f),s\leq \ell(f)}\Pr{R=m,L=s}\cdot 2^{-(m-s+1)}\right) ^{\alpha } \\
&\leq (2n+1)\sum_{f}\sum_{m\geq r(f),s\leq \ell(f)}\Pr{R=m,L=s}^{\alpha }\cdot 2^{-(m-s+1)\alpha },
\end{align*}
where the second inequality again uses $(p+q)^\alpha \leq q^\alpha+p^\alpha$.

We exchange the order of summation and note that for each $m,s$ there are $2^{m-s+1}$ terms, to get
\begin{align*}
\sum_{f}\sum_{x}\mu ^{(n)}(f,x)^{\alpha } &\leq (2n+1)\sum_{m,s}\Pr{R=m,L=s}^{\alpha }2^{(m-s+1)(1-\alpha) }.
\end{align*}
Since this sum has $(n+1)^2$ summands,
\begin{align*}
\sum_{f}\sum_{x}\mu ^{(n)}(f,x)^{\alpha }
&\leq 4n^3\max_{m,s}\Pr{R=m,L=s}^{\alpha }2^{(m-s+1)(1-\alpha)} \\
&\leq 4n^3\max_{k} \Pr{R-L+1=k}^\alpha \cdot 2^{k(1-\alpha)}.
\end{align*}
Together with (\ref{eq:llineq}), we have
\begin{equation}
    \label{eq:lleq}\log \sum_{g}\mu ^{(n)}(g)^{\alpha } 
=\max_{k}\log \left[ \Pr{R-L=k}^\alpha \cdot 2^{k(1-\alpha)}\right]+o(n).
\end{equation}

We claim that we can replace $\Pr{R-L=k}$ in the equation above with $\Pr{R-L\geq k}$. On the one hand, we trivially have $\Pr{R-L=k}\leq \Pr{R-L\geq k}$, yielding one direction. For the other direction, we note that for any $k'\geq 0$ there exists $k^*\geq k'$ such that
\[\Pr{R-L=k^*}\geq \frac{1}{n}\Pr{R-L\geq k'}\]
by the pigeonhole principle. Therefore, for any $k'$ there is a $k^*\geq k'$ such that  
\begin{align*}
   \max_k\log \left[\Pr{R-L=k}^\alpha \cdot 2^{k(1-\alpha)}\right] 
   &\geq \log \left[\Pr{R-L=k^*}^\alpha \cdot 2^{k^*(1-\alpha)}\right]\\
   &\geq \log \left[\frac{1}{n^\alpha}\Pr{R-L\geq k'}^\alpha \cdot 2^{k'(1-\alpha)}\right]\\
   &= \log \left[\Pr{R-L\geq k'}^\alpha \cdot 2^{k'(1-\alpha)}\right]+\log \frac{1}{n^\alpha}.
\end{align*}
Taking the supremum over $k'$ yields 
\[\max_k\log \left[\Pr{R-L=k}^\alpha \cdot 2^{k(1-\alpha)}\right] \geq \max_k\log \left[\Pr{R-L\geq k}^\alpha \cdot 2^{k(1-\alpha)}\right]+o(n). \]
Now (\ref{eq:lleq}) becomes 
\[\log \sum_{g}\mu ^{(n)}(g)^{\alpha } 
=\max_{k}\log \left[\Pr{R-L\geq k}^\alpha\cdot 2^{k(1-\alpha)}\right]+o(n).\]

Citing Theorem 1 and Remark 1 in \cite{HamanaKesten2002simple}, we have that for any $x\in [0,1]$
\[\lim_n \frac{1}{n}\log\Pr{R-L\geq nx}=\psi(x),\]where $\psi(x)=-\frac{1}{2}(1+x)\log(1+x)-\frac{1}{2}(1-x)\log(1-x).$
We note that when a sequence of decreasing functions converges pointwise to a continuous function, the sequence also converges uniformly. Therefore, we can write 
\[\frac{1}{n}\log\Pr{R-L\geq k}=\psi(k/n)+o(1),\] where $o(1)$ vanishes uniformly in $k$. We have
\begin{align*}
    \frac{1}{n}\log \sum_{g}\mu ^{(n)}(g)^{\alpha } 
&=\frac{1}{n}\max_{k}\log \left[\Pr{R-L\geq k}^\alpha \cdot 2^{k(1-\alpha)}\right]+o(1)\\
&=\frac{1}{n}\max_{k}\left[\alpha\log \Pr{R-L\geq k}+k(1-\alpha)\log2\right] +o(1)\\
&=\max_{k}\left[\alpha\psi(\frac{k}{n})+\frac{k}{n}(1-\alpha)\log2\right] +o(1).
\end{align*}

Finally, let $p:=\frac{k}{n}$. Then 
\begin{align*}
   h_\alpha(\mu)&=\frac{1}{1-\alpha} \lim  \frac{1}{n}\log \sum_{g}\mu ^{(n)}(g)^{\alpha }\\
   &=\frac{1}{1-\alpha} \lim_n \max_{p\in \{0/n, 1/n,...,n/n\}}\left[\alpha \psi(p)+p(1-\alpha)\log2\right]\\
   &= \max_{p\in [0,1]}\left[\frac{\alpha}{1-\alpha} \psi(p)+p\log2\right]\\
   &=\max_{p\in [0,1]} \varphi_\alpha(p)
\end{align*}
where
\begin{align*}
    \varphi_\alpha(p) 
    &=\frac{\alpha}{1-\alpha}  \psi(p)+p\log2\\
    &= p \log 2 - \frac{\alpha}{2(1-\alpha)}\left[ (1 - p) \log(1-p) + 
    (1 + p) \log(1 + p)\right].
\end{align*}
Some computations show that the above expression is maximized when 
\begin{equation*}
    p_{\alpha }^{\ast }=\frac{4^{1/\alpha}-4}{4^{1/\alpha}+4}.
\end{equation*}
The continuity of the R\'enyi entropy from the left extends the result to $h_1$. This completes the proof of Theorem~\ref{thm:sws}.
\end{proof}

\subsection{The asymptotic min-entropy}

The next claim provides an example of a non-symmetric random walk on an amenable group for which the asymptotic min-entropy $h_\infty$ is positive. 
Consider a drifting SWS walk $\mu_\beta$ on $\Z_2\wr \Z$. That is, let $\mu_\beta$ be given by $\mu_\beta = \eta * \sigma * \eta$, where $\eta(f^1,0)=\eta(0,0) = 1/2$, and $\sigma(0,-1)=\frac{1}{2}(1-\beta)=1-\sigma(0,1)$ for some $\beta \in (-1,1)$. The drift of the walker is $\sum_g \mu_\beta(g)\pi(g) = \beta$.
\begin{claim}
\label{clm:h-infty-amenable}
    If $\beta \neq 0$ then $h_\infty(\mu_\beta)>0$.
\end{claim}
\begin{proof}
    By symmetry, we can assume without loss of generality that $\beta>0$. Let $Z_1,Z_2,\ldots$ be the $\mu_\beta$-random walk. Then, by the Chernoff bound, there is some $r>0$ such that $\Pr{\pi(Z_n) < n\beta/2} \leq \ee^{-rn}$. In particular, if $\pi(g) \leq n\beta/2$ then $\mu_\beta(g) \leq \ee^{-rn}$.
    
    Now, conditioned on $\pi(Z_n)=k \geq 0$, it holds by the definition of the SWS walk that the lamp configuration restricted to $\{0,1,\ldots,k-1\}$ is distributed uniformly, so that each configuration has probability $2^{-k}$. Thus, for $k \geq n\beta/2$ and $g$ such that $\pi(g)=k$, 
    \begin{align*}
        \mu_\beta(g)\leq 2^{-n\beta/2}=\ee^{-\frac{1}{2}\log(2)\beta n}.
    \end{align*}

    It follows that if we let $c = \min\{\frac{1}{2}\log(2)\beta,r\}$, then $\mu_\beta(g) \leq \ee^{-cn}$ for all $g \in G$, and $h_\infty(\mu) \geq c$.
\end{proof}

Next, we show that for non-degenerate random walks on non-amenable groups the asymptotic min-entropy is always positive. Note that this follows from Kesten's Theorem for symmetric $\mu$.
\begin{claim}
\label{clm:h-infty-nonamenable}
    Let $\mu$ be a finitely supported, non-degenerate probability measure on a non-amenable group $G$. Then $h_\infty(\mu) > 0$.
\end{claim}
To prove this claim we recall some basic definitions. Let $\ell^2(G)$ be the Hilbert space of real square-integrable functions on $G$, equipped with the standard inner product and norm. For $h \in G$, let $R_h \colon \ell^2(G) \to \ell^2(G)$ be the right shift operator given by $[R_h(\varphi)](g) = \varphi(g h)$. Then $R_g$ is an orthogonal linear operator and $h \mapsto R_h$ is the right regular representation of $G$. Let $M = \sum_h \mu(h)R_h$ be the Markov operator of the $\mu$-random walk. As is well known (see, e.g., \cite[Theorem 12.5]{woess2000random}) when $G$ is non-amenable and $\mu$ is non-degenerate, then the operator norm $\rho(M)$ is strictly less than $1$, i.e., $\rho(M) := \sup \{\Vert M\varphi\Vert \,:\, \Vert \varphi \Vert = 1\} < 1$. Note that this holds even when $\mu$ is not symmetric. 
\begin{proof}[Proof of Claim~\ref{clm:h-infty-nonamenable}]
    Note that $[M^n \varphi](g) = \sum_h \mu^{(n)}(h)\varphi(g h)$, and so $[M^n\delta_e](g) = \mu^{(n)}(g^{-1})$. Since $\rho(M^n) \leq  \rho(M)^n$, $\Vert M^n \delta_e \Vert \leq \rho(M)^n$, and so
    \begin{align*}
        \rho(M)^{2n} \geq \sum_g \mu^{(n)}(g^{-1})^2 \geq \max_g \mu^{(n)}(g)^2.
    \end{align*}
    Hence
    \begin{align*}
        -\log \rho(M) \leq -\frac{1}{n}\log\max_g \mu^{(n)}(g).
    \end{align*}
    Taking the limit as $n$ tends to infinity yields that
    \begin{align*}
        -\log\rho(M) \leq h_\infty(\mu),
    \end{align*}
    and so, since $\rho(M) < 1$, we have proved the claim.
\end{proof}

\bibliography{refs}
\end{document}